\documentclass[]{birkjour}
\usepackage{amsthm}
\usepackage{amsmath}
\usepackage{amsfonts}
\usepackage{amssymb}
\usepackage{verbatim}
\usepackage[all,cmtip]{xy}
\usepackage{color}
\usepackage{comment}
\usepackage{enumerate}
\newtheorem{thm}{Theorem}[section]

\newtheorem{cor}[thm]{Corollary}

\newtheorem{remark}[thm]{Remark}
\newtheorem{lemma}[thm]{Lemma}
\newtheorem{prop}[thm]{Proposition}

\newtheorem{defn}[thm]{Definition}

 \numberwithin{equation}{section}
\newcommand{\bb}[1]{\mathbb{#1}}
\newcommand{\cl}[1]{\mathcal{#1}}

\newcommand{\bs}[1]{\boldsymbol{#1}}

\newcounter{egcounter}
\begin{document}
\title[Extending De Branges and Beurling]{Multivariable Sub-Hardy Hilbert Spaces Invariant under the action of $n$-tuple of Finite Blaschke factors}

\author{Sneh Lata}
\address{Department Of Mathematics\\
         Shiv Nadar University\\
         School of Natural Sciences,\\
         Gautam Budh Nagar - 203207\\
         Uttar Pradesh, India}
\email{sneh.lata@snu.edu.in}

\author{Sushant Pokhriyal}
\address{Department Of Mathematics\\
         Shiv Nadar University\\
         School of Natural Sciences\\
         Gautam Budh Nagar - 203207\\
         Uttar Pradesh, India}
\email{sp259@snu.edu.in}

\author{ Dinesh Singh}
\address{Centre For Lateral Innovation, Creativity and Knowledge\\
        SGT University\\
        Gurugram 122505\\
        Haryana, India}
\email{dineshsingh1@gmail.com}



\begin{abstract} This paper deals with representing in concrete fashion those Hilbert spaces that are vector subspaces of the Hardy spaces 
$H^p(\bb D^n) \ (1\le p\le \infty)$ that remain invariant under the action of coordinate wise multiplication by an $n$-tuple $(T_{B_1},\dots, T_{B_n})$ of operators where each 
$B_i, \ 1\le i\le n,$ is a finite Blaschke factor on the open unit disc. The critical point to be noted is that these $T_{B_i}$ are assumed to be weaker than isometries as operators. Thus our main theorems extends the principal result of \cite{LS} in the following three directions: $(i)$ from one to several variables; $(ii)$ from multiplication with the coordinate function $z$ to an $n$-tuple of multiplication by finite Blaschke factors $B_i, \ 1\le i\le n;$ $(iii)$ from vector subspaces of $H^2(\bb D)$ to the case of vector subspaces of 
$H^p(\bb D^n), \ 1\le p\le \infty.$ We further derive a generalization of Slocinski's well known Wold type decomposition of a pair of doubly commuting isometries to the case of $n$-tuple of doubly commuting operators whose actions are weaker than isometries. 
\end{abstract}

\maketitle

\section{Introduction} 
One of the cornerstones of several approaches to describing the structure of invariant subspaces of some well known operators relies on the device first used fruitfully by Halmos while describing the invariant subspaces of the shift operator of infinite multiplicity. This was essentially the use of what is now referred to as the Wold decomposition. We refer to \cite{Hal, Hel}. The theory has evolved way beyond the classical Wold decomposition and has used to great advantage the central idea of a wandering subspace. This has led to breakthroughs-using deep technical arguments which have since been simplified a fair 
amount-in describing the invariant subspaces of the operator of multiplication by the coordinate function $z$ on the 
Bergman space and the Dirichlet space  of analytic functions on the open unit disc $\bb D$, see \cite{A, ARS, R}.  

The thrust behind the use of wandering subspace stems from relying on conditions on the operator of multiplication by $z$ that are far more general than requiring the operator to be an isometry. Some authors such as Izuchi  \cite{Izu} and Shimorin \cite{Shim} have produced a decomposition of Hilbert spaces where the operator of multiplication by 
$z$ satisfies certain conditions which encompass the requirement of being an isometry.  
  
The first and the third authors of this paper adopted an approach in \cite{LS} that also resulted in a description of Hilbert spaces that are vector subspaces of the classical Hardy space $H^2(\bb D)$ and on which the operator of multiplication by the coordinate function $z$ satisfies certain conditions that also encompass the requirement of being an isometry. These conditions are independent of the conditions of Shimorin in \cite{Shim}. Indeed, in \cite{LS} we give an example of a periodic weighted shift which satisfies our conditions as enunciated in \cite{LS} but does not 
satisfy the conditions of \cite{Shim}. In fact, we also have an example of an operator that does not satisfy our condition as given in \cite{LS} but satisfies the conditions from \cite{Shim}. We give details of these examples in Section 2 where we define a near-isometry which is simply an operator that satisfies the conditions of \cite{LS}. The additional advantage of our result in 
\cite{LS} is that it generalizes some well known results such as the scalar case of de Branges' generalization of Beurling's invariant subspace theorem \cite{Beu} and also the main result of \cite{SS2}. Interested readers can see \cite{Sar2} and \cite{SS2} for  simpler poofs of de Branges theorem for the scalar case.  

In this paper by overcoming some redoubtable technical obstacles, we have managed to adopt an approach along the lines of \cite{LS} to characterize a class of Hilbert spaces on $\bb D^n \ (n\ge 1)$ that are invariant under the action of multiplication by multivariable versions of finite Blaschke products. The conditions imposed on these operators are far weaker than the classical assumptions of being isometric.  

Our main features in this work below claim novelty on four counts: 
\begin{enumerate}[(i)]
\item We generalize-in Theorem 3.1-the main theorem of \cite{LS} to a multivariable situation where the action of multiplication by the coordinate function $z$ as in \cite{LS} is replaced by us in our Theorems\ \ref{Hp-bdd-sv} and 
Theorem \ref{p>2} to multiplication by a multivariable version of a finite Blaschke product acting by an $n$-tuple of finite Blaschke products acting on the Hardy spaces over $\bb D^n.$  

\item In the simplest case of $n=1,$ our Theorem \ref{Hp-bdd-sv} and Theorem \ref{p>2} are generalizations of the main theorem of \cite{LS} since the operator of multiplication by $z$ is replaced with multiplication by an arbitrary finite Blaschke product.

\item In addition, we have demonstrated our results to be valid across all Hardy spaces $H^p(\bb D^n)$ for $1\le p\le \infty$ and not just $H^2(\bb D^n).$

\item Finally, we have derive a multivariable generalization of the Slocinski's two variable Wold decomposition for doubly commuting tuple of isometries \cite{Slo} under much weaker assumptions than the requirement of isometricity.
\end{enumerate} 

The organization of the paper is as follows. We first introduce the concept of a ``near-isometry". We then produce a decomposition (Theorem \ref{wdinv}) for a near-isometry which reduces to the classical Wold decomposition when the near-isometry is actually an isometry. We note here that our decomposition  is outside the purview of Shimorin's decomposition \cite{Shim} which is also a generalization of the Wold decomposition. 

Proceeding in the same vein we produce a generalization of Slocinski's well known two variable Wold decomposition for doubly commuting tuples of isometries \cite{Slo} to the class of doubly commuting $n$-tuples of near-isometries (Theorem \ref{wdinvsv}). Indeed, our decomposition generalizes the $n$-variable extension 
\cite[Theorem 3.1]{Sar} of Slocinski's decomposition. This decomposition is of interest in its own right and at the same time it allows us to extend Theorem 3.1 from \cite{LS} and Theorem 4.1 from \cite{ST} in a far reaching manner to give representation of all Hilbert spaces that are algebraically (boundedly) contained in $H^p(\bb D^n), p\ge 2 \ (1\le p<2)$ on which an $n-$tuple of the operators of multiplication by finite Blaschke factors doubly commute and act as near-isometries (Theorem \ref{Hp-bdd-sv} and \ref{p>2}).  

We want to remind the reader that Theorem 3.1 from \cite{LS} and Theorem 4.1 from \cite{ST} do not assume any 
topological condition between the sub-Hardy Hilbert space and the Hardy space $H^2(\bb D),$ whereas our Theorem 
\ref{Hp-bdd-sv}, as mentioned above, assumes bounded containment of the sub-Hardy Hilbert space in $H^p(\bb D^n)$ when $1\le p<2.$ However, as we record in Remark \ref{n=1}, this containment condition in our result is redundant 
for all $p$ in the one-variable case, that is, $n=1.$ Thus, Theorem \ref{Hp-bdd-sv} and Theorem \ref{p>2} for $n=1$ 
generalize Theorem 3.1 from \cite{LS} and Theorem 4.1 from \cite{ST} without adding any extra hypothesis. 
Further, as mentioned in Remark \ref{n=1}, when $p=2,$ the bounded containment condition in Theorem \ref{Hp-bdd-sv} is redundant for all $n.$ Hence, we would like to record here that we have, in the spirit of \cite{Singh}, \cite{SS2}, and \cite{ST} dropped any topological connections between the sub-Hardy Hilbert space and the one-variable Hardy space $H^p(\bb D)$ for any $1\le p\le \infty,$ and the multivariable Hardy spaces $H^p(\bb D^n)$ for all $p\ge2.$   

\section{Notations and preparatory results}
For $n\ge 1,$ let $\bb D^n$  be the open unit polydisc in $\bb C^n, \ \bb T^n$ be the $n$-torus, and let $m$ be the 
normalized $n$-dimensional Lebesgue measure on $\bb T^n.$  Further, let $L^p(\bb T^n)$ and 
$H^p(\bb T^n), \ 1\le p\le \infty,$ denote the familiar Lebesgue spaces on $\bb T^n.$ It is well known that using the Poisson kernel for the polydisc $\bb D^n,$ $H^p(\bb T^n)$ can be identified 
with the Hardy space $H^p(\bb D ^n).$ Henceforth, whenever required, we will consider $H^p(\bb D^n)$ as a closed subspace of $L^p(\bb T^n)$ without any mention. Recall that for $1\le p<\infty, \ H^p(\bb D ^n)$ is a Banach space consisting of analytic functions 
on $\bb D^n$ with norm 
\begin{equation}
||f||_p= \left(\sup_{0\le r<1}\int_{\bb T^n} |f(re^{i\theta_1},\dots,re^{\theta_n})|^p dm\right)^{1/p} <\infty
\end{equation}
and $H^{\infty} (\bb D^n)$ is a Banach algebra consisting of bounded analytic functions on $\bb D^n$ with the supremum norm. In fact, $H^2(\bb D^n)$ is a Hilbert space with the above-mentioned norm.  

A $k$-tuple ${\bf T}=(T_1,\dots,T_k)$ of commuting operators on a Hilbert space $\cl H$ is said to be doubly commuting if $T_iT_j^*=T_j^*T_i$ whenever $i\ne j,$ and a subspace $\cl M$ is said to be joint $\bf T$-invariant(reducing) if $\cl M$ is invariant(reduces) under each $T_i.$ Let ${\bf T} = (T_1,\dots,T_k)$ be a doubly commuting tuple of operators on a Hilbert space $\cl H.$ For an integer $m\in \{1,\dots,k\},$ we define $I_m=\{1,\dots, m\},$ and for $A=\{i_1,\dots,i_l\}\subseteq I_m,$ we define ${\bf T}_A=(T_{i_1},\dots, T_{i_l})$ and 
$\bb N_{0}^A=\bb N_0\times \dots\times \bb N_0$ ($|A|$ copies of $\bb N_0=\bb N\cup \{0\}$). 
For ${\bs k}=(k_{i_1},\dots,k_{i_l})\in \bb N_0^A,$ we denote ${\bf T}^{\bs k}_A=T_{i_1}^{k_1}\cdots T_{i_l}^{k_l}.$
Further, we define $\cl W_i=\cl H\ominus T_i\cl H=Ker(T_i^*)$ and $\cl W_A=\bigcap\limits_{i\in A}\cl W_i.$ In case 
$A=\emptyset,$ we take 
$\cl W_A=\cl H.$ 

\begin{defn} Let $\cl H$ be a Hilbert space. An operator $T\in B(\cl H)$ is said to be a {\bf {\em near-isometry}} if it satisfies the following two conditions:
\begin{enumerate}
\item[(i)] There exists a constant $\delta>0$ such that $\delta ||x|| \le ||Tx||\le ||x||$ for all $x\in \cl H$.
\item[(ii)] For each $k\ge 1, \ T^{*k}T^{k+1}{\cl H}\subseteq T\cl H.$
\end{enumerate} 
\end{defn} 

We note that the conditions imposed on the operator of multiplication with $z$ in \cite{LS} are precisely the conditions for being a near-isometry. In \cite{Shim}, Shimorin produced a decomposition of Hilbert spaces $\cl H$ connected with operators $T$ that satisfy one of the following two conditions:
$$
||T^2x||^2 +||x||^2 \le 2 ||Tx||^2 \ \ {\rm for \ any} \ x\in \cl H
$$
or
$$
||Tx+y||^2 \le 2\left(||x||^2 + ||Ty||^2\right) \ \ {\rm for \ any}\  x, y\in \cl H
$$

As we promised in the introduction, we will now give examples to establish that our conditions as written in the definition of a near-isometry are independent of Shimorin's conditions. Let $\beta_n=\frac{1}{2^{\frac{n}{2}}}$ when $n$ is even, and $\beta_n=\frac{1}{2^{\frac{n-1}{2}}}$ when $n$ is odd. Then $T_z,$ multiplication by $z,$ on the weighted Hardy space $H^2(\beta)$ is clearly a near-isometry, but it does not satisfy any of the two above-stated 
conditions of Shimorin. We will now give an example of an operator that is not a near-isometry, but satisfy one of the Shimorin's conditions. We have taken it from \cite[Example 4.4]{Jab} where the author constructed it to study hyperexpansive composition operators.   
 
Let $X=\{(n,m)\in \bb Z\times \bb Z: n\le m\}$ and let $\{a_n\}_{n\in \bb Z}$ be a sequence of natural numbers such that $a_n=1$ if $n\ge 1$ and $a_n=2$ if $n\le 0.$ Further, let $\mu$ be measure on the power set of $X$ defined as 
$\mu(\{(n,n)\})=1$ for $n\in \bb Z$ and $\mu(\{(n,m)\})=a_n$ for $n<m.$ Set 
$$
e_{i,j}=\left\{
\begin{array}{ccc}
\frac{1}{\sqrt{a_i}}\chi_{\{(i,j)\}} & {\rm if} & i<j \\
\chi_{\{(i,j)\}} & {\rm if} & i=j
\end{array}\right.
$$

Then $\{e_{i,j}: (i,j)\in X\}$ is an orthonormal basis for $L^2(X,\mu)$ and  
$$
T(e_{i,j})=\left\{
\begin{array}{ccc}
e_{i,j+1} & {\rm if} & i<j \\
e_{i+1,i+1}+\sqrt{a_i} e_{i,i+1}  & {\rm if} & i=j.
\end{array}\right.
$$
defines a bounded linear operator on $L^2(X,\mu).$ Simple computations yield 
$$
||T^2f||^2+||f||^2\le 2 ||Tf||^2
$$
 for all $f\in L^2(X,\mu)$ and hence $T$ satisfies one of the Shimorin's conditions. We will now show that 
 $T$ is not a near-isometry. For this, we will show that there exists a natural number $k$ such that 
 $T^{*k}T^{k+1} (L^2)\nsubseteq T( L^2)$ which is equivalent to showing that 
 $T^k(N)\not\perp T^{k+1}(L^2),$ where $N=L^2\ominus T(L^2).$  Let $f=e_{1,2}-e_{2,2}.$ Then $f\in N$ and 
 $T(f)=e_{1,3}-e_{3,3}-e_{2,3}\in T(N).$ If we now take $g=e_{1,1}\in L^2$, then 
 $T^2(g)=e_{3,3}+e_{2,3}+e_{1,3}\in T^2(L^2).$ Clearly, $T(f)\not\perp T^2(g)$ which shows that 
 $T(N)\not\perp T^2(L^2).$ Thus, $T^*T^2(L^2)\nsubseteq T(L^2).$ Hence, T is not a near-isometry. 
 
\begin{prop}\label{wdinv-prop} Let $(T_1,\dots,T_k)$ be a $k$-tuple of doubly commuting near-isometries, and let $A$ be a 
subset of $I_k.$ Then for each $j\in I_k\setminus A$
\begin{enumerate} 
\item[(i)] $\cl W_A$ reduces $T_j,$
\item[(ii)] $\cl W_A\ominus T_j\cl W_A=\cl W_A\bigcap W_j,$
\item[(iii)] $T_j$ is a near-isometry on $\cl W_A.$
\end{enumerate}
\end{prop}
\begin{proof} Fix $j\in I_k\setminus A.$ To show that $\cl W_A$ is invariant under $T_j,$ let $x\in \cl W_A.$ Then $x\in \cl W_i$ for each $i\in A$ which implies that $T_i^*(x)=0$ for each $i\in A$. Thus, $T_i^*\left(T_jx\right)=T_j\left(T^*_ix\right)=0$ for each $i\in A.$ 
Therefore, $T_j(x)\in \cl W_A$ which shows that $\cl W_A$ is invariant under $T_j.$ By using similar arguments and the fact that 
each $T_i$ commutes with $T_j,$ we deduce that $\cl W_A$ is invariant under $T_j^*.$ Hence, $\cl W_A$ reduces $T_j.$ 

To prove $(ii)$, let $x\in \cl W_A\bigcap W_j.$ Then $T_j^*(x)=0$ which implies, $\langle{T_j^*x,h}\rangle=\langle{x,T_jh}\rangle=0$ for all $h\in \cl W_A$ which 
proves that 
$x\in \cl W_A\ominus T_j\cl W_A.$ To prove the other containment, let $x\in \cl W_A\ominus T_j\cl W_A.$ Then 
$\langle{T_j^*x,h}\rangle= \langle{x,T_jh}\rangle=0$ for all $h\in \cl W_A.$ This implies that $T_j^*x\perp \cl W_A.$ But 
$T_j^*(x)\in \cl W_A,$ as $\cl W_A$ is invariant under $T_j^*.$ Therefore, $T_j^*(x)=0$ which establishes that $x\in \cl W_A\cap W_j,$ and thereby completes the proof of $(ii)$. 

To prove $(iii)$, we first note that $T_j$ is clearly 
a contraction and bounded below on $\cl W_A.$ Now, using 
the facts that $\cl W_A$ reduces $T_j$ and $T_j$ is a near-isometry on $\cl H,$ we conclude that $T_j^{*m}T_j^{m+1}(\cl W_A)\subseteq T_j(\cl W_A)$ for all $m\ge 0.$ Lastly, since $\cl W_A$ reduces $T_j,$ therefore $T_j^{*m}T_j^{m+1}(\cl W_A)\subseteq T_j(\cl W_A)$ for all $m\ge 0$ even when we consider $T_j$ as an operator on $\cl W_A.$ Hence, $T_j$ is a near-isometry on $\cl W_A.$
\end{proof}

The following simple observations from operator theory are extremely handy for our proof of Theorem \ref{wdinvsv}. 

\begin{lemma}\label{ot} Let $T_1\in B(\cl H)$ be a bounded below operator. Suppose $\cl W$ is a closed subspace of $\cl H$ and $T_2\in B(\cl H)$ is an injective operator such that
\begin{enumerate}
\item[(i)] $T_1T_2=T_2T_1,$
\item[(ii)] $T_1(\cl W)\subseteq \cl W,$
\item[(iii)] $T_2^k(W)\perp T_2^m(\cl H)$ whenever $0\le k<m.$
\end{enumerate}
Then the following equalities hold:
\begin{enumerate}
\item[(i)] $T_1\left(\bigoplus\limits_{m=0}^\infty T_2^m\cl W\bigoplus \bigcap\limits_{m=0}^\infty T_2^m\cl H\right) = \bigoplus\limits_{m=0}^\infty T_1T_2^m\cl W\bigoplus \bigcap\limits_{m=0}^\infty T_1T_2^m\cl H$,
\item[(ii)] $T_1\left(\bigcap\limits_{m=0}^\infty T_2^m\cl R\right) = \bigcap\limits_{m=0}^\infty T_1T_2^m\cl R$ 
 \ for \ any \ subspace \ $R$ \ of \ $\cl H$,
\item[(iii)] $\bigcap\limits_{m_1=0}^\infty T_1^{m_1}\left(\bigoplus\limits_{m_2=0}^\infty T_2^{m_2}\cl W\right)
=\bigoplus\limits_{m_2=0}^\infty T_2^{m_2}\left(\bigcap\limits_{m_1=0}^\infty T_1^{m_1}\cl W\right)$
\end{enumerate}
\end{lemma}

\section{Wold decomposition  for doubly commuting tuples of near-isometries}
A closed subspace $\cl W$ of a Hilbert space $\cl H$ is said to be a wandering subspace for a bounded linear operator $T$ on 
 $\cl H$ if $T^{k}\cl W\perp T^l\cl W$ whenever $k\ne l.$ Further, an operator $T\in B(\cl H)$ is said to be a shift on $\cl H$ if there exists 
a wandering subspace $\cl W$ of $T$ such that
$$
\cl H = \bigoplus\limits_{k=0}^\infty T^k\cl W.
$$
%

\begin{thm}{\bf (Wold decomposition theorem)}\cite[Page 109]{Hof} Let $T$ be an isometry on a Hilbert space $\cl H.$ Then we can decompose $\cl H$ into two reducing subspaces for $T$ as $\cl H=\cl H_s\oplus \cl H_u$ such that $T|_{\cl H_s}$ is a shift and $T|_{\cl H_u}$ is a unitary. Moreover,  
$$
\cl H_s = \bigoplus\limits_{k=0}^\infty T^k{\cl W} \quad and \quad \cl H_u = \bigcap\limits_{k=0}^\infty T^k{\cl H} 
$$ 
where $\cl W=\cl H\ominus T\cl H$
\end{thm}

The Wold decomposition theorem plays a vital role in may areas of operator theory and operator algebras, in particular, invariant subspace problem for Hilbert spaces of holomorphic functions. It has greatly simplified the proofs of many classical results, like Beurling's theorem, its generalization due to de Branges, and many of their generalizations; thereby it has generated a new perspective to the area. 

In \cite{Slo}, Slocinski gave a Wold type decomposition for pairs of doubly commuting isometries which facilitated a de Branges theorem for the bidisc \cite{Singh, Red}. Recently, Sarkar \cite{Sar} generalized Slocinski's decomposition to several variables and used it in \cite{SSW} to obtain a vector-valued version of Beurling's theorem for the polydisc. 

\begin{thm}\cite[Theorem 3.1]{Sar}\label{sar} Let $V=(V_1,\dots, V_l)$ be an $l$-tuple ($l\ge 2$) of doubly commuting isometries on a Hilbert space 
$\cl H.$ Then for every $m\in\{2,\dots,l\},$ there exist $2^m$ joint $(V_1,\dots,V_m)-$reducing subspaces $\{\cl H_A:A\subseteq I_m\}$ such that 
$$
\cl H=\bigoplus _{A\subseteq I_m}\cl H_A.
$$
Moreover, for each $A\subseteq I_m$
$$
\cl H_A=\bigoplus_{\bs k\in \bb N_0^A}V_A^{\bs k}\left(\bigcap\limits_{\bs j\in \bb N_0^{I_m\setminus A}}V_{I_m\setminus A}^{\bs j}\cl W_A\right),
$$
and whenever $\cl H_A\ne \{0\}, \ V_i|_{\cl H_A}$ is a shift for $i\in A$ and unitary for $i\in I_m\setminus A.$ 
\end{thm}

In this paper, we work with near-isometries which may not necessarily be isometries, and so we do not have the Wold decomposition or it's above-mentioned extensions at our disposal. Interestingly, we establish an analogue (Theorem \ref{wdinv}) of Wold decomposition for near-isometries. Furthermore, following the ideas of Sarkar \cite{Sar}, we extend this decomposition to doubly commuting tuples of near-isometries (Theorem \ref{wdinvsv}) which, 
in fact, generalizes the above-mentioned Sarkar's decomposition (Theorem \ref{sar}) to our setting.

The first and the third authors proved the following analogue of the Wold decomposition for near-isometries implicitly in the proof of their main theorem in \cite{LS}.
We would like to state and prove it here as an independent result as apart from being interesting in its own right it also acts as a foundation for our analogue of the Wold decomposition for doubly commuting tuples of near-isometries.

 \begin{thm}\label{wdinv} Let $T\in B(\cl H)$ be a near-isometry. Then $\cl H$ can be decomposed into two reducing subspaces of $T$ as 
$$\cl H=\cl H_0\oplus \cl H_1$$ such that $T$ is a shift on $H_0$ and invertible on $H_1.$ Moreover,  
$$
H_0=\bigoplus_{m=0}^{\infty} T^m\cl W, \quad \cl H_1= \bigcap_{m=0}^{\infty}T^m\cl H.
$$
where $\cl W=\cl H\ominus T\cl H$
\end{thm}
\begin{proof} Let $\cl W=\cl H\ominus T\cl H.$ We are given that $T$ is a near-isometry, therefore $T^{*m}T^{m+1}\cl H\subseteq T\cl H$ for all $m\ge 0$ which imply that $T^m\cl W\perp T^{m+1}\cl H$ for all $m\ge 0.$ Thus, $T^m\cl W\perp T^k\cl W$ whenever $m\ne k.$ Also, $T$ is bounded below, so each $T^m\cl W$ is a closed subspace of $\cl H.$ Let 
$\cl H_0=\bigoplus\limits_{m=0}^\infty T^m\cl W$ and $\cl H_1=\bigcap\limits_{m=0}^\infty T^m\cl H.$ We will show 
$\cl H_0^\perp = H_1.$ 

Since $T^m\cl W\perp T^{m+1}\cl H$ for all $m\ge 0,$ therefore $\cl H_0\perp \cl H_1$ which implies that $\cl H_1\subseteq \cl H_0^\perp.$ For the other containment, observe that $T^m\cl W\perp T^{m+1}\cl H$ for all $m\ge 0$ which implies that 
$T^m\cl H=T^m\cl W\oplus T^{m+1}\cl H$ for every $m\ge 0.$ Thus, if $x\perp T^m\cl W$ for all $m\ge 0,$ then $x\in T^m\cl H$ for all 
$m\ge 0.$ Therefore, $H_0^{\perp} = \cl H_1.$  

Lastly, by definition, $\cl H_0$ and $\cl H_1$ are both $T$-invariant, $T$ is a shift on $\cl H_0,$ and invertible on $\cl H_1.$ This completes the proof. 
\end{proof}
%

\begin{thm}\label{wdinvsv} Let $(T_1,\dots,T_k)$ be a doubly commuting $k$-tuple of near-isometries on a Hilbert space 
$\cl H.$ Then for any $m, \ 2\le m\le k,$ there exist $2^m$ joint $(T_1,\dots,T_m)$- reducing subspaces $\{\cl H_A:A\subseteq I_m\}$ 
such that 
$$
\cl H=\bigoplus _{A\subseteq I_m} \cl H_A
$$
where for a non-empty $A\subseteq I_m,$
\begin{equation}\label{rep1}
\cl H_A=\bigoplus_{{\bs r}\in \bb N_0^A} T^{\bs r}_A\left(\bigcap_{{\bs j}\in \bb N_0^{I_m\setminus A}}T^{\bs j}_{I_m\setminus A}\cl W_A\right),
\end{equation}
 and for $A=\emptyset,$
 \begin{equation}\label{empty}
 \cl H_A = \bigcap_{{\bs r}\in \bb N_0^{I_m}} T^{\bs r}_{I_m }\cl H.
 \end{equation}
Furthermore, for each $A\subseteq I_m,$ $T_i|_{\cl H_A}$ is a shift if $i\in A$ and is invertible if $i\in I_m\setminus A.$
\end{thm}

\begin{proof} We will prove the result by induction on $m,$ and for this, we first prove the result for $m=2.$ Since $T_1$ is a near-isometry on $\cl H,$ therefore using Theorem 
\ref{wdinv}, we decompose $\cl H$ as 

\begin{equation}\label{wdinvsv1}
\cl H=\bigoplus_{m_1=0}^\infty T_1^{m_1}\cl W_1\bigoplus \bigcap_{m_1=0}^\infty T_1^{m_1}\cl H,
\end{equation}
where $\cl W_1=\cl H\ominus T_1\cl H, \ T_1 |_{\bigoplus_{m_1=0}^\infty T_1^{m_1}\cl W}$ 
is a shift, and $T_1|_{\bigcap_{m_1=1}^\infty T_1^{m_1}\cl H}$ is invertible. By Proposition \ref{wdinv-prop}, $T_2\cl W_1\subseteq \cl W_1$ and $T_2$ is a near-isometry on $\cl W_1,$ therefore applying  Theorem \ref{wdinv} for $\cl W_1$ we can write 
\begin{align}
\cl W_1&=\bigoplus_{m_2=0}^\infty T_2^{m_2}(\cl W_1\ominus T_2\cl W_1)\bigoplus \bigcap_{m_2=0}^\infty T_2^{m_2}\cl W_1\nonumber\\
&=\bigoplus_{m_2=0}^\infty T_2^{m_2}(\cl W_1\cap\cl W_2)\bigoplus \bigcap_{m_2=0}^\infty T_2^{m_2}\cl W_1\label{wdinvsv2}
\end{align}

Using the decomposition of $\cl W_1$ from equation (\ref{wdinvsv2}) in equation(\ref{wdinvsv1}), we obtain 

\begin{align}
\cl H &=\bigoplus_{m_1=0}^\infty T_1^{m_1}\left( \bigoplus_{m_2=0}^\infty T_2^{m_2}(\cl W_1\cap \cl W_2) \bigoplus 
\bigcap_{m_2=0}^\infty T_2^{m_2}\cl W_1\right)\bigoplus \bigcap_{m_1=0}^\infty T_1^{m_1} \cl H\nonumber\\
&=\bigoplus_{m_1,m_2=0}^\infty T_1^{m_1} T_2^{m_2}(\cl W_1\cap \cl W_2) \bigoplus 
\left(\bigoplus_{m_1=0}^\infty T_1^{m_1}\left(\bigcap_{m_2=0}^\infty  T_2^{m_2}\cl W_1\right)\right)\nonumber\\
& \quad \quad \bigoplus \bigcap_{m_1=0}^\infty T_1^{m_1} \cl H\label{wdinvsv8}
\end{align} 
The second equality above follows using Lemma \ref{ot}. Note that  $T_2$ is also a near-isometry on $\cl H,$ therefore using Theorem \ref{wdinv} again, we can decompose $\cl H$ as 
$$
\cl H = \bigoplus_{m_2=0}^\infty T_2^{m_2}\cl W_2 \bigoplus \bigcap_{m_2=0}^\infty T_2^{m_2}\cl H
$$
Then, using Lemma \ref{ot}, we get 
$$
T_1^{m_1}\cl H = \bigoplus_{m_2=0}^\infty T_1^{m_1}T_2^{m_2}\cl W_2 \bigoplus \bigcap_{m_2=0}^\infty T_1^{m_1}T_2^{m_2}\cl H.
$$
This yields 
\begin{align}
\bigcap_{m_1=0}^\infty T_1^{m_1}\cl H &=\bigcap_{m_1=0}^\infty \left(\bigoplus_{m_2=0}^\infty T_1^{m_1}T_2^{m_2}\cl W_2\right) 
\bigoplus \bigcap_{m_1=0}^\infty \left(\bigcap_{m_2=0}^\infty T_1^{m_1}T_2^{m_2}\cl H\right)\nonumber\\
&= \bigoplus_{m_2=0}^\infty T_2^{m_2}\left( \bigcap_{m_1=0}^\infty T_1^{m_1}\cl W_2 \right)
\bigoplus \bigcap_{m_1,m_2=0}^\infty  T_1^{m_1}T_2^{m_2}\cl H\label{wdinvsv3}
\end{align}

Using equation (\ref{wdinvsv3}) in equation (\ref{wdinvsv8}), we get 
\begin{align*}
\cl H &= \bigoplus_{m_1,m_2=0}^\infty T_1^{m_1} T_2^{m_2}(\cl W_1\cap \cl W_2) \bigoplus 
\left(\bigoplus_{m_1=0}^\infty T_1^{m_1}\left(\bigcap_{m_2=0}^\infty  T_2^{m_2}\cl W_1\right)\right)\\
& \quad \quad \bigoplus \left(\bigoplus_{m_2=0}^\infty T_2^{m_2}\left( \bigcap_{m_1=0}^\infty T_1^{m_1}\cl W_2\right) \right)
\bigoplus \bigcap_{m_1,m_2=0}^\infty  T_1^{m_1}T_2^{m_2}\cl H
\end{align*}

Thus, for $m=2,$ $\cl H$ is represented as a direct sum of closed subspaces where the summands in the above decomposition of $\cl H$ are precisely the same as the the ones written in equations (\ref{rep1}) and (\ref{empty}).  

Clearly, each summand $\cl H_A$ in the above decomposition of $\cl H$ is invariant under $T_1$ as well as $T_2.$ Thus, each summand reduces $T_1$ as well as $T_2.$ Now for $A=I_2, \ \cl H_A= \bigoplus\limits_{m_1,m_2=0}^\infty T_1^{m_1} T_2^{m_2}(\cl W_1\cap \cl W_2).$ 

But, 
\begin{align*}
\bigoplus_{m_1,m_2=0}^\infty T_1^{m_1} T_2^{m_2}(\cl W_1\cap \cl W_2)
&=\bigoplus_{m_1=0}^\infty T_1^{m_1}\left(\bigoplus_{m_2=0}^\infty T_2^{m_2}(\cl W_1\cap \cl W_2)\right)\\
&=\bigoplus_{m_2=0}^\infty T_2^{m_2}\left(\bigoplus_{m_1=0}^\infty T_1^{m_2} (\cl W_1\cap \cl W_2)\right)
\end{align*}
Therefore, $T_1$ and $T_2$ both are shifts on $\cl H_A.$  

Now consider $A=\{1\}.$ Then $\cl H_A=\bigoplus\limits_{m_1=0}^\infty T_1^{m_1}\left(\bigcap\limits_{m_2=0}^\infty  T_2^{m_2}\cl W_1\right).$ Then clearly $T_1$ is a shift on $\cl H_A.$ We already have that $\cl H_A$ is invariant under $T_2$ and $T_2$ is one-to-one. Also, $T_2\left( \bigcap\limits_{m_2=0}^\infty  T_2^{m_2}\cl W_1 \right)= \bigcap\limits_{m_2=0}^\infty  T_2^{m_2}\cl W_1$ which 
implies that $T_2(\cl H_A)=\cl H_A.$ Thus, $T_2$ is invertible on $\cl H_A.$  

\vspace{.2 cm}

Similarly when $A=\{2\}, \ T_1$ is invertible and $T_2$ is a shift on $\cl H_A= \bigoplus\limits_{m_2=0}^\infty T_2^{m_2}\left( \bigcap\limits_{m_1=0}^\infty T_1^{m_1}\cl W_2 \right).$  

Lastly, when $A$ is empty, $\cl H_A=\bigcap\limits_{m_1,m_2=0}^\infty  T_1^{m_1}T_2^{m_2}\cl H$ in which case, using the arguments similar as above, it can be easily seen that $T_1$ and $T_2$ are both invertible on $\cl H_A.$  

Hence, the result is true for $m=2.$ Now suppose the result is true for $m=r$ where $r+1\le k.$ Then
\begin{equation}\label{wdinvsv3*}
\cl H=\bigoplus_{A\subseteq I_r} \cl H_A,
\end{equation}
where for each non-empty $A\subseteq I_r,$
\begin{equation}\label{wdinvsv4}
\cl H_A= \bigoplus_{{\bs l}\in \bb N_0^A} T_A^{\bs l}\left(\bigcap_{{\bs j}\in \bb N_0^{I_r\setminus A}} T_{I_r\setminus A}^{\bs j}\cl W_A\right)
\end{equation}
and for $A=\emptyset$
\begin{equation}\label{wdinvsv5}
\cl H_A=\bigcap_{{\bs l}\in \bb N_0^r}T_{I_r}^{\bs l}(\cl H). 
\end{equation}
Let $A$ be a non-empty subset of $I_r.$ Since $T_{r+1}$ is a near-isometry on $\cl W_A,$ therefore using Theorem \ref{wdinv}, we can write 
$$
\cl W_A=\bigoplus_{m_{r+1}=0}^\infty T_{r+1}^{m_{r+1}}\left(\cl W_A\cap \cl W_{r+1}\right)
\bigoplus \bigcap_{m_{r+1}=0}^\infty T_{r+1}^{m_{r+1}}\cl W_A
$$
Using this representation of $\cl W_A$ in equation (\ref{wdinvsv4}) and Lemma \ref{ot}, we get 
\begin{align}\label{wdinvsv5*}
  \cl H_A  
&= \bigoplus_{{\bs l}\in \bb N_0^A, m_{r+1}\in \bb N_0} T^{\bs l}_A T^{m_{r+1}}_{r+1}\left(\bigcap_{{\bs j}\in \bb N_0^{I_r\setminus A}}T^{\bs j}_{I_r\setminus A}\left(\bigcap_{i\in A\cup\{r+1\}}\cl W_{i}\right)\right)\nonumber\\ 
& \quad \quad \bigoplus\left( \bigoplus_{{\bs l}\in \bb N_0^A}T^{\bs l}_A\left(\bigcap\limits_{\bs j\in \bb N_0^{I_r\setminus A}}\bigcap\limits_{m_{r+1}=0}^\infty T^{\bs j}_{I_r\setminus A}T_{r+1}^{m_{r+1}}\cl W_A\right)\right).
\end{align}
Further, when $A=\emptyset,$

\begin{equation}\label{wdinvsv6}
\cl H_A=\bigcap_{{\bs l}\in \bb N_0^r} T^{\bs l}_{I_r}\cl H 
\end{equation} 
Now applying Theorem \ref{wdinv} for the near-isometry $T_{r+1}$ on $\cl H,$ we have 
$$
 \cl H = \bigoplus_{m_{r+1}=0}^\infty T_{r+1}^{m_{r+1}}\cl W_{r+1} \bigoplus \bigcap_{m_{r+1}=0}^\infty \cl H
 $$
Using this representation of $\cl H$ in equation (\ref{wdinvsv6}), we get 
\begin{align}\label{wdinvsv7}
\cl H_A &= \bigcap_{{\bs l}\in \bb N_0^r} T_{I_r}^{\bs l} \left( \bigoplus_{m_{r+1}=0}^\infty T_{r+1}^{m_{r+1}}\cl W_{r+1} 
\bigoplus \bigcap_{m_{r+1}=0}^\infty \cl H \right)\nonumber\\
& =  \bigcap_{{\bs l}\in \bb N_0^r} \left( \bigoplus_{m_{r+1}=0}^\infty  T_{I_r}^{\bs l} T_{r+1}^{m_{r+1}}\cl W_{r+1} \bigoplus \bigcap_{m_{r+1}=0}^\infty T_{I_r}^{\bs l} T_{r+1}^{m_{r+1}}\cl H \right)\nonumber\\
& =  \bigoplus_{m_{r+1}=0}^\infty T_{r+1}^{m_{r+1}}\left( \bigcap_{{\bs l}\in \bb N_0^r} T_{I_r}^{\bs l} \cl W_{r+1}\right) \bigoplus \bigcap_{{\bs l}\in \bb N_0^{I_{r+1}}} T_{I_{r+1}}^{\bs l} \cl H
\end{align}

Using equations (\ref{wdinvsv5*}) and (\ref{wdinvsv7}) in equation (\ref{wdinvsv3*}), we see that $\cl H=\bigoplus\limits_{A\subseteq I_{r+1}}\cl H_A,$ where the summands $\cl H_A$ are same as in equations (\ref{rep1}) and (\ref{empty}) for each subset $A$ of $I_{r+1.}$ Also, by similar arguments as were used in the case for $m=2,$ it can be seen that $T_i|_{\cl H_A}$ is shift whenever $i\in A$ 
and is invertible whenever $i\in I_{r+1}\setminus A.$ This completes the proof.
\end{proof}

\begin{remark} If the operators $T_1,\dots, T_k$ in Theorem \ref{wdinvsv} are assumed to be isometries, then Theorem \ref{wdinvsv} reduces to Theorem 3.2. Hence, Sarkar's decomposition (Theorem 3.2) is a special case of our decomposition; thereby Slocinski's decomposition \cite{Slo} also turns out to be a special case of our decomposition. 
\end{remark}

\section{de Branges' theorem for doubly commuting $n$-tuples of near-isometries}

In this section, we give a representation of sub-Hardy Hilbert spaces over the polydisc $\bb D^n$ which are invariant under any $k$-tuple of operators of multiplication with finite Blaschke products that are near-isometries and doubly commute on the sub-Hardy Hilbert space. We first get a representation when the involved multiplication operators are assumed to be isometries, and later use it to work with the general case.

Let $B_1, \dots, B_n$ be finite Blaschke products with degrees $r_1,\dots, r_n$, respectively. We know from \cite{ST} that each finite Blaschke product give rise to an orthonormal basis for $H^2(\bb D).$ We first recall this construction of basis for $H^2(\bb D)$ generated using the finite Blaschke $B_i.$ Let 
$$
B_i(z)=\prod_{l=1}^{r_i}\frac{z-\alpha_l^i}{1-\overline{\alpha_l^i}z}
$$
For $1\le j\le r_i,$ define 
$$
k_j^i(z)=\frac{1}{1-\overline{\alpha_j^i}z}, \quad  \widehat{k_j^i}(z)=\frac{\sqrt{1-|\alpha_j^i|^2}}{1-\overline{\alpha_j^i}z}, \quad and \quad B_j^i(z)= \prod_{l=1}^{j}\frac{z-\alpha_l^i}{1-\overline{\alpha_l^i}z}.
$$
Further, define 
\begin{equation}\label{onb}
e_{jm}^i=\widehat{k^i_{j+1}} B_j^i B_i^m,
\end{equation}
for $0\le j\le r_i-1$ and $m\ge 0,$ where $B_0^i(z)=1$

Then $\{e_{jm}^i: 0\le j\le r_i-1, \ m\ge 0 \}$ forms an orthonormal set in $H^2(\bb D)$ which becomes an orthonormal basis for $H^2(\bb D)$ if we assume $\alpha_1^i=0.$ 

From this point onwards, $B_1,\dots, B_n$ are fixed finite Blaschke products with $r_1,\dots, r_n$ factors, respectively. For each $i,$ let $T_{B_i}:H^p(\bb D^n)\to H^p(\bb D^n)$ denote the isometry 
$$
T_{B_i}f(z_1,\dots,z_n)=B_i(z_i)f(z_1,\dots,z_n).
$$
 
For any automorphism $\phi$ of unit disc and any fixed $1\le i\le n$, if we let $\phi_i:\bb D^n\to \bb D^n$ denote the map  $\phi_i(z_1,\dots,z_n)=(z_1,\dots,z_{i-1},\phi(z_i),z_{i+1},\dots,z_n),$ then the composition operator induced by $\phi_i$ on the Hardy space $H^p(\bb D^n), \ 1\le p\le \infty$ is well-defined and one-to-one. Therefore, in our investigation of representation of a Hilbert space that is a vector subspaces of a Hardy space $H^p(\bb D^n)$ on which the tuple $(T_{B_1},\dots,T_{B_n})$ doubly commute and each $T_{B_i}$ is a well-defined near-isometry, we can without loss of any generality assume that $\alpha_1^i=0$ for each $1\le i\le n.$ 

Then, for every Blaschke product $B_i,$ the set $\{e_{jm}^i: 0\le j\le r_i-1, \ m\ge 0 \}$ is an orthonormal basis for $H^2(\bb D).$ 
Now for ${\bs j}=(j_1,\dots, j_n)$ with $0\le j_i\le r_i-1$ and ${\bs m}=(m_1,\dots, m_n)\in \bb N_0^n$, define 
$$
e_{{\bs j}{\bs m}}(z_1,\cdots,z_n) = e^1_{j_1m_1}(z_1)\cdots e^n_{j_nm_n}(z_n).
$$
Then, using the identification of $H^2(\bb D^n)$ with 
$H^2(\bb D)\otimes \cdots \otimes H^2(\bb D),$  $\{e_{{\bs j}{\bs m}}: 0\le j_i\le r_i-1, {\bs m} \in \bb N_0^n\}$ forms an orthonormal basis for $H^2(\bb D^n)$ and 
\begin{equation}\label{ST7}
H^2(\bb D^n) = \bigoplus_{j_1,\dots,j_n=0}^{r_1-1,\dots,r_n-1}e_{j_10}^1\cdots e_{j_n0}^n H^2(B_1,\cdots B_n) 
\end{equation}
where $H^2(B_1,\dots,B_n)$ is the closed linear span of $\{B_1^{m_1}\cdots B_n^{m_n}:m_i\ge 0\}$ in 
$H^2(\bb D^n).$ Following the terminology of 
\cite{ST}, we call a scalar $\alpha_{{\bs j}{\bs m}}$ the $({\bs j},{\bs m})^{th} \ B-$Fourier coefficient of $f\in H^2(\bb D^n)$ if $\alpha_{{\bs j}{\bs m}}=\langle{f,e_{{\bs j}{\bs m}}}\rangle.$ 

\begin{lemma}\label{STL1} Let $\phi\in H^\infty(\bb D^n).$ Then each $\phi_{\bs i}\in H^\infty(\bb D^n),$ where 
$\phi_{\bs i}'s$ are unique functions in 
$H^2(B_1,\cdots, B_n)$ such that 
$$
\phi=\sum\limits_{\substack{{\bs i}=(i_1,\cdots,i_n)\\i_1,\cdots, i_n=0}}^{r_1-1,\cdots, r_n-1 } e_{i_10}^1\cdots e_{i_n0}^n \phi_{\bs i}.
$$ 
\end{lemma}

\begin{proof} To prove the result, it is enough to show that each $\phi_i$ multiplies $H^2(B_1,\dots, B_n)$ into $H^2(\bb D^n)$. For this, we take $f\in H^2(B_1,\dots,B_n),$ and let $f=\sum\limits_{i_1,\dots,i_n =0}^\infty \alpha_{i_1\cdots i_n}B_1^{i_1}\cdots B_n^{i_n}.$ Suppose  
$f_{k_1\dots k_n}=\sum\limits_{i_1,\dots,i_n= 0}^{k_1,\cdots, k_n} \alpha_{i_1\cdots i_n}B_1^{i_1}\cdots B_n^{i_n}.$ Then 
$f_{k_1\dots k_n}$ converges to $f$ in $H^2(\bb D^n).$ Now consider,
\begin{align}
||\phi_{\bs i} f_{k_1\dots k_n}||^2 &\le \sum\limits_{\substack{{\bs j}=(j_1,\dots, j_n)\\ j_1,\dots, j_n= 0}}^{r_1-1,\dots, r_n-1}||\phi_{\bs j} f_{k_1\dots k_n}||^2\nonumber\\
&= \sum\limits_{\substack{{\bs j}=(j_1,\dots, j_n)\\ j_1,\dots, j_n=0}}^{r_1-1,\dots, r_n-1}||e_{j_10}^1\cdots e_{j_n0}^n\phi_{\bs j} f_{k_1\dots k_n}||^2\nonumber\\
&= \|\phi f_{k_1\dots k_n} \|^2\label{ST8}
\end{align}

Since $\{f_{k_1\dots k_n}\}$ is Cauchy in $H^2(\bb D^n)$ and $\phi\in H^\infty(\bb D^n), \  \{\phi f_{k_1\dots k_n}\}$ is also Cauchy in 
$H^2(\bb D^n).$ This, using inequality (\ref{ST8}), implies that $\{\phi_i f_{k_1\dots k_n}\}$ is Cauchy, hence converges to some 
$h\in H^2(\bb D^n).$ But, convergence in $H^2(\bb D^n)$ implies pointwise convergence; therefore $\phi_i f = h.$ Thus, we conclude that $\phi_i$ multiplies 
$H^2(B_1,\dots, B_n)$ into $H^2(\bb D^n).$ This, using equation (\ref{ST7}) together with the fact that each 
$e_{j_10}^1\cdots e_{j_n0}^n$ is in $H^\infty(\bb D^n),$ establishes that $\phi_i$ multiplies the entire $H^2(\bb D^n)$ into itself. Hence, $\phi_i$ belongs to $H^\infty(\bb D^n.)$    
\end{proof}

From this point onwards, we will work with the $n$-tuple $(T_{B_1},\dots, T_{B_n})$ only and, for notational convenience, we will use $T_i$ in place of $T_{B_i}$ in all our results. The following is an extension of Theorem 4.1 from \cite{ST}. 

\begin{thm}\label{STThm} Let $\cl M$ be a non-zero Hilbert space that is a vector subspace of $H^p(\bb D^n)$ for some $1\le p\le 2,$ and let 
$\cl M$ be boundedly contained in $H^p(\bb D^n)$ whenever $1\le p<2.$ Suppose the operators $T_1,\dots,T_n$ satisfy the following conditions on $\cl M$:
\begin{enumerate}[(i)]
\item $T_j(\cl M)\subseteq \cl M$ for each $j=1,\dots, n;$
\item $T_j$ is an isometry on $\cl M$ for each $j=1,\dots, n;$
\item $T_jT_m^*=T_m^*T_j$ for every $j\ne  m,$ where the adjoint is with respect the inner product on $\cl M.$
\end{enumerate}

Then there exits an orthonormal set $\{\phi_1,\dots,\phi_r\}$ in $\cl M$ with $r\le r_1\cdots r_n$ consisting of $H^{\frac{2p}{2-p}}(\bb D^n)$ 
functions ($\frac{2p}{2-p}$ means $\infty$ when $p=2$) 
such that
\begin{equation}\label{rep3}
\cl M=\phi_1H^2(B_1,\dots,B_n)\oplus \cdots\oplus \phi_r H^2(B_1,\dots,B_n)
\end{equation} and 
\begin{equation}\label{rep4}
||\phi_1f_1+\cdots+\phi_rf_r||_{\cl M}^2=||f_1||_2^2+\cdots+||f_r||_2^2
\end{equation}
whenever $f_1,\dots,f_r\in H^2(B_1,\dots,B_n).$ 
\end{thm}

The following two lemmas are essentially part of the proof of Theorem \ref{STThm}, but we are proving them separately for reader's convenience. 
We note in advance that the idea of the proof of Lemma \ref{STL2} for the case $p=2$ is same as Lemma 4.2 from \cite{ST}. Also, the proof of Lemma \ref{STL3} 
is essentially a careful extension of the basic idea used in the proof of Lemma 4.3 from \cite{ST} to the several variable situation. We still present these proofs below for completeness.

\begin{lemma}\label{STL2} Let $\cl M$ be a non-zero Hilbert space. Suppose $\cl M$ and the 
operators $T_1,\dots, T_n$ satisfy the conditions of Theorem \ref{STThm}. If $\phi \in \cl M$ such that 
$\{\phi B_1^{m_1}\cdots B_n^{m_n} : m_1,\dots, m_n\ge 0\}$ is an orthonormal set in $\cl M,$ then
\begin{enumerate}
\item[(i)] $\phi H^2(B_1,\dots, B_n)\subseteq \cl M;$
\item[(ii)] $\phi \in H^\frac{2p}{2-p}(\bb D^n),$ where $\frac{2p}{2-p}$ means $\infty$ for $p=2;$
\item[(iii)] $||\phi f||_{\cl M}=||f||_2$ for all $f\in H^2(B_1,\dots, B_n).$
\end{enumerate}
\end{lemma}

\begin{proof} To prove $(i)$, let $f\in H^2(B_1,\dots,B_n)$ and $f=\sum\limits_{i_1,\dots,i_n= 0}^\infty \alpha_{i_1\cdots i_n}B_1^{i_1}\cdots B_n^{i_n}.$ Suppose $f_{k_1\dots k_n}=\sum\limits_{i_1,\dots,i_n\ge 0}^{k_1,\dots, k_n} \alpha_{i_1\cdots i_n}B_1^{i_1}\cdots B_n^{i_n}.$ Then $f_{k_1\dots k_n}$ converges to $f$ in $H^2(\bb D^n).$ Now consider
\begin{align}
||\phi f_{k_1\cdots k_n}||^2_{\cl M} &= ||\sum\limits_{i_1,\dots,i_n= 0}^{k_1\cdots k_n} \alpha_{i_1\cdots i_n}\phi B_1^{i_1}\cdots B_n^{i_n}||_{\cl M}^2\nonumber\\
&=  \sum\limits_{i_1,\dots,i_n= 0}^{k_1\cdots k_n} ||\alpha_{i_1\cdots i_n}\phi B_1^{i_1}\cdots B_n^{i_n}||_{\cl M}^2\nonumber\\
&=  \sum\limits_{i_1,\dots,i_n= 0}^{k_1\cdots k_n} |\alpha_{i_1\cdots i_n}|^2\nonumber\\
&= ||f_{k_1\cdots k_n}||^2_2\label{ST9}
\end{align}
Thus, $\{\phi f_{k_1\cdots k_n}\}$ is a Cauchy sequence in $\cl M$ and hence converges to some $h\in \cl M.$ We divide the rest of the proof of part $(i)$ in the following two case. 

\vspace{2 mm}

\noindent {\bf \underline{Case: p=2}} For any fixed ${\bs k}=(k_1,\dots,k_n)\in \bb N_0^n,$ we can write 
\begin{align*}
h &= \sum\limits_{\substack{{\bs j}=(j_1,\dots, j_n)\\ j_1,\dots, j_n = 0}}^{k_1,\dots, k_n}\alpha_{\bs j}\phi B_1^{j_1}\cdots B_n^{j_n}+B_1^{k_1+1}h_1+\cdots+ B_n^{k_n+1}h_n\\
&= \phi f_{k_1\dots k_n}+ B_1^{k_1+1}h_1+\cdots+ B_n^{k_n+1}h_n\
\end{align*}
for some $h_1,\dots,h_n\in \cl M.$ Let $\phi=\sum\limits_{\substack{{\bs i}=(i_1,\cdots,i_n)\\i_1,\cdots, i_n=0}}^{r_1-1,\cdots, r_n-1 } e_{i_10}^1\cdots e_{i_n0}^n \phi_{\bs i}$ for some $\phi_{\bs i}\in H^2(B_1,\dots, B_n).$ Then the coefficient of $e_{{\bs i}{\bs k}}$ in $h$ is given by
\begin{align*}
\langle{h,e_{{\bs i}{\bs k}}} \rangle 
&= \langle {\phi f_{k_1\dots k_n}, e_{{\bs i}{\bs k}}}\rangle\\
&= \langle {\phi_i f_{k_1\dots k_n}, B_1^{k_1}\cdots B_n^{k_n} }\rangle
\end{align*}

This shows that $\phi_i f\in H^2(B_1,\dots, B_n)$ which implies that $\phi f\in H^2(\bb D^n).$ We also get 
$\langle{\phi f, e_{{\bs i}{\bs k}}}\rangle = \langle {h, e_{{\bs i}{\bs k}}}\rangle$ which implies $h=\phi f.$ Thus, 
$\phi H^2(B_1,\dots,B_n)\subseteq \cl M.$ 

\vspace{2 mm}

\noindent {\bf \underline{Case: $1\le p<2$}} Since $\cl M$ is boundedly contained in $H^p(\bb D^n)$ and $\phi f_{k_1\cdots k_n}$ 
converges to $h$ in $\cl M$, therefore $\phi f_{k_1\cdots k_n}$ 
converges to $h$ in the $p$-norm. Also, $f_{k_1\cdots k_n}$ converges to $f$ in $2$-norm. Now using the fact that convergence in $p$-norm as well as the convergence in $2$-norm implies pointwise convergence on $\bb D^n,$ we conclude $h=\phi f.$  Hence, $\phi f \in \cl M.$  

Thus, from the above two cases, we conclude that $\phi H^2(B_1,\dots, B_n)\subseteq \cl M.$ This establishes $(i)$.

Now using $(i)$ and equation (\ref{ST7}) together with the fact that each $e_{i_10}^1\cdots e_{i_n0}^n$ is in $H^\infty(\bb D^n)$, we readily conclude that 
$\phi H^2(\bb D^n) \subseteq H^p(\bb D^n).$ Hence $\phi \in H^\infty (\bb D^n)$ when $\cl M\subseteq H^2(\bb D^n)$ and 
$\phi \in H^\frac{2p}{2-p} (\bb D^n)$ when $\cl M\subseteq H^p(\bb D^n)$ for $1\le p<2.$ Lastly, $(iii)$ can be deduced easily from equation (\ref{ST9}).
\end{proof}

\begin{lemma}\label{STL3} Let $\cl M$ be a non-zero Hilbert space that is a vector subspace of $H^2(\bb D^n)$. Suppose $\cl M$ and the 
operators $T_{B_1}, \dots, T_{B_n}$ satisfy the hypotheses of Theorem \ref{STThm}. If $\phi_1, \dots, \phi_r$ are non-zero functions in 
$H^\infty(\bb D^n)$ such that
\begin{enumerate}
\item[(i)] $\phi_i H^2(B_1,\dots, B_n)\subseteq \cl M$ for each $i=1,\dots, r,$
\item[(ii)] $\phi _i H^2(B_1,\dots, B_n)\perp \phi _j H^2(B_1,\dots, B_n)\ $ in $\cl M$ whenever $i\ne j.$
\end{enumerate}
Then $r\le r_1\cdots r_n.$
\end{lemma}
\begin{proof} To understand the proof better we assume that each $B_i$ has 2 factors, that is, $r_1= \dots= r_n=2.$ The proof for the general case is identical.  

Suppose there are $r_1\cdots r_n+1 = 2^n+1$ non-zero functions $\phi_1, \dots, \phi_{2^n+1}$ in 
$\cl M$ which satisfy the hypotheses $(i)$ and $(ii)$ given in the statement of the lemma. Using the decomposition of $H^2(\bb D^n)$ given by 
equation (\ref{ST7}), we can write 
\begin{equation}\label{ST10}
\phi_j = \sum\limits_{\substack{i_1,\dots, i_n=0\\{\bs i}=(i_1,\dots,i_n)}}^{1,\dots, 1}e^1_{i_10}\cdots e_{i_n0}^n\phi_{\bs i}^j, \quad 1\le j\le 2^n+1.
\end{equation}
Let ${\bs t}_1,\cdots,{\bs t}_{2^n}$ be the enumeration of the set $\{{\bs i}=(i_1,\dots,i_n): 0\le i_j\le 1\}$ such that 
${\bs t_1}<\cdots<{\bs t_{2^n}}$ (dictionary order) and define 
$$
A_1 = 
\left(\begin{array}{cccc}
\phi_1 & \phi_2    & \cdots                  & \phi_{2^n+1}\\
\phi_{\bs t_1}^1 & \phi_{\bs t_1}^2 & \cdots & \phi_{\bs t_1}^{2^n+1}\\
\vdots                  & \vdots                  & \ddots  & \vdots \\
\phi_{\bs t_{2^n}} ^1 & \phi_{\bs t_{2^n}}^2 & \dots & \phi_{\bs t_{2^n}}^{2^n+1} 
\end{array}\right),
$$
Then, using equation (\ref{ST10}), $det(A_1)=0.$ Note that each $\phi_j\in H^\infty(\bb D^n),$ therefore, using Lemma \ref{STL1} , each $\phi_{\bs t_i}^j\in H^\infty(\bb D^n).$ Also, $\phi_{\bs t_i}^j\in H^2(B_1,\dots, B_n).$ Therefore, the minor of each $\phi_j$ is a well-defined element of $H^2(B_1,\dots, B_n).$ Let $\lambda_j^1$ denote the minor of $\phi_j.$ Then 
$$
det(A_1)=\sum_{j=1}^{2^n+1} (-1)^{j-1}\phi_j \lambda_j^1=0.
$$
Therefore 
$$
\left\langle{\sum_{j=1}^{2^n+1} (-1)^{j-1}\phi_j \lambda_j^1,\phi_k\lambda_k^1}\right\rangle_{\cl M} = 0
$$
for every $1\le k\le 2^n+1.$ Now using the hypothesis $(ii),$ we get 
$\langle{\phi_k\lambda_k^1,\phi_k\lambda_k^1}\rangle_{\cl M}=0$ which implies $\lambda_k^1=0$ for every $1\le k\le 2^{n+1}.$ 

For the next step, take any $2^n$ elements $p_1, \dots, p_{2^n}\subseteq \{1,\dots, 2^{n}+1\}$ such that $p_1<p_2<\cdots<p_{2^n}$ (usual order) and $2^n-1$ elements ${\bs q}_1,\dots, {\bs q}_{2^n-1}\subseteq \{{\bs i}=(i_1,\dots, i_n): 0\le i_j\le 1\}$ such that  ${\bs q}_1<{\bs q}_2<\cdots {\bs q}_{2^n-1}$ (dictionary order) and form the $2^n\times 2^n$ matrix
$$
A_2=
\left(\begin{array}{cccc}
\phi_{p_1} & \phi_{p_2}    & \cdots                  & \phi_{p_{2^n}}\\
\phi_{{\bs q}_1}^{p_1} & \phi_{{\bs q}_1}^{p_2} & \cdots & \phi_{{\bs q}_1}^{p_2^n}\\
\vdots                  & \vdots                  & \ddots  & \vdots \\
\phi_{{\bs q}_{2^n-1}} ^{p_1} & \phi_{{\bs q}_{2^n-1}}^{p_2} & \dots & \phi_{{\bs q}_{2^n-1}}^{p_n} 
\end{array}\right),
$$
By using the expansion of $\phi_{p_j}$ as written in equation (\ref{ST10}), we obtain that $det(A_2)$ is equal to a scalar (either 1 or -1) multiple of $\lambda_k^1$ for some $1\le k\le 2^n+1,$ hence must be zero.

Again, as explained above, the minor of each $\phi_{p_j}$ is a well-defined element of $H^2(B_1,\dots, B_n),$ let us denote it by 
$\lambda_{j}^2.$ Then, 
\begin{equation}\label{ST11}
det(A_2)=\sum_{i=1}^{2^n}(-1)^{j-1}\phi_{p_j}\lambda_{j}^2=0.
\end{equation}

Now using equation (\ref{ST11}), hypothesis $(ii)$, and following the similar arguments as above, we obtain that $\lambda_{p_j}^2=0$ for $1\le j\le 2^n.$ We repeat this process for all  choices of 
$\{p_1,\dots, p_{2^n} \}\subseteq \{1,\dots, 2^n+1\}$ and 
$\{{\bs q}_1,\dots, {\bs q}_{2^n-1}\}\subseteq \{{\bs i}=(i_1,\dots, i_n): 0\le i_j\le 1\}.$   

After $2^n-2$ iterations, we form a $3\times 3$ matrix 
$$
A_{2^n-1}=
\left(\begin{array}{ccc}
\phi_{p_1} & \phi_{p_2} & \phi_{p_3}\\
\phi_{{\bs q}_1}^{p_1} & \phi_{{\bs q}_1}^{p_2} & \phi_{{\bs q}_1}^{p_3}\\
\phi_{{\bs q}_2}^{p_1} & \phi_{{\bs q}_2}^{p_2} & \phi_{{\bs q}_2}^{p_3}
\end{array}\right)
$$
for some $p_1, p_2, p_3\in \{1,\dots, 2^n+1\}$ with $p_1<p_2<p_3$ (usual order) and 
${\bs q}_1, {\bs q}_2\in \{{\bs i}=(i_1,\dots, i_n): 0\le i_j\le 1\}$ with ${\bs q}_1<{\bs q}_2$ (dictionary order). 
Then, as done before, using a minor from the $(2^n-2)^{th}$ step we obtain that $det (A_{2^n-1})=0.$ Let $\lambda_i^{2^n-1}$ denote the minor of $\phi_{p_i}$ which is a well-defined element of $H^2(B_1,\dots, B_n)$ because of Lemma \ref{STL1} and equation (\ref{ST10}). Now using the hypothesis $\phi_i H^2(B_1,\dots B_n)\perp \phi_j H^2(B_1,\dots, B_n)$ whenever $i\ne j$ one more time,  we obtain that $\phi_{{\bs q}_1}^{p_i}\phi_{{\bs q}_2}^{p_j}-\phi_{{\bs q}_1}^{p_j}\phi_{{\bs q}_2}^{p_i}=0$ for $i\ne j, \ i, j=1,2,3.$ By repeating this for all possible choices of $p_1,p_2,p_3$ and ${\bs q}_1,{\bs q}_2,$ we obtain
\begin{equation}\label{ST12*}
\phi_{\bs j}^l\phi_{\bs k}^m = \phi_{\bs k}^l \phi_{\bs j}^m
\end{equation}
for any $l, m\in\{1,\dots 2^{m+1}\}$ with $l\ne m$ and ${\bs j}, {\bs k}\in \{{\bs i}=(i_1,\dots, i_n): 0\le i_l\le 1\}$ with 
${\bs j}\ne {\bs k}.$

Lastly, for any fixed choice of $p_1, p_2 \in \{1,\dots, 2^n+1 \}$ with $p_1<p_2$ and 
${\bs q}\in \{{\bs i}=(i_1,\dots, i_n): 0\le i_j\le 1\}$ form the matrix
$$
A_{2^n}=
\left(\begin{array}{cc}
\phi_{p_1} & \phi_{p_2}\\
\phi_{\bs q}^{p_1} & \phi_{\bs q}^{p_2}
\end{array}\right)
$$ 
Then $det(A_{2^n})= \sum\limits_{\substack{i=(i_1,\dots, i_n)\\{\bs i}\ne {\bs q}\\ i_1\,\dots,i_n=0}}^{1,\dots, 1}e_{i_10}^1\cdots e_{i_n0}^n(\phi_{\bs i}^{p_1}\phi_{\bs q}^{p_2}-\phi_{\bs q}^{p_1}\phi_{\bs i}^{p_2})=0,$ using equation (\ref{ST12*}). Thus, $\phi_{p_1}\phi_{\bs q}^{p_2}=\phi_{p_2}\phi_{\bs q}^{p_1}=0.$ 
But $\phi_{p_1}H^2(B_1,\dots, B_n)\perp \phi_{p_2}H^2(B_1,\dots, B_n).$ Therefore $\phi_{p_1}\phi_{\bs q}^{p_2}=0=\phi_{p_2}\phi_{\bs q}^{p_1}$ which implies that 
$\phi_{\bs q}^{p_1}=\phi_{\bs q}^{p_2}=0.$ Repeating the process for choices of $p_1$ and ${\bs q},$ we conclude that $\phi_{\bs q}^j=0$ for all $1\le j\le 2^n+1$ and ${\bs q}\in \{{\bs i}=(i_1,\dots, i_n):0\le i_j \le 1\}$ which implies that $\phi_j=0$ for all $j,$ but each $\phi_j$ is non zero. Hence we arrive at a contradiction which establishes that we can't have more than $2^n$ functions in $H^\infty(\bb D^n)$ that 
satisfy the hypotheses $(i)$ and $(ii)$. This completes the proof.  
\end{proof}

Before giving the proof of Theorem \ref{STThm}, we note that it's proof for the Case $p=2,$ is essentially Lemmas \ref{STL2} and \ref{STL3}; therefore, it is indeed a careful extension of ideas from \cite{ST} to several variables situation. But, we would like to mention that the ideas used for the Case $1\le p<2$ are new and original.

 \vspace{.2 cm}
 
\noindent{\bf \underline{Proof of Theorem \ref{STThm}}} Since $T_1,\dots, T_n$ acts as isometries on $\cl M$ and doubly commute as operators on $\cl M,$ therefore using Theorem \ref{wdinvsv}, we can decompose $\cl M$ as  
\begin{equation}\label{rep9}
\cl M= \bigoplus_{A\subseteq I_n} \cl M_A,
\end{equation}
where 
\begin{align*}
\cl M_A&=\bigoplus_{{\bs r}\in \bb N_0^A} T^{\bs r}_A\left(\bigcap_{{\bs j}\in \bb N_0^{I_n\setminus A}}T^{\bs j}_{I_n\setminus A}(\cl W_A)\right) \quad when 
\quad  \emptyset \ne A\subseteq I_n\\
and \quad  \cl M_A &= \bigcap_{{\bs r}\in \bb N_0^{I_n}} T^{\bs r}_{I_n }\cl M \quad when \quad A=\emptyset
\end{align*}
with $\cl W_A=\bigcap_{i\in A}\cl W_i, \ \cl W_i=ker(T_{i}^*).$ Note that elements of $\cl M$ are analytic functions on $\bb D^n$ and 
$B_i(0)=0$ for all $1\le i\le n$, therefore $\cl M_A=\{0\}$ whenever $I_n\setminus A \ne \emptyset.$ So, equation (\ref{rep9}) reduces to  
\begin{equation}\label{rep2*}
\cl M = \bigoplus_{{\bs r}\in I_n}T^{\bs r}_{I_n}\cl W_{I_n}
\end{equation}

We will now show that the dimension of $\cl W_{I_n}$ can at be most $r_1\cdots r_n.$ For this we will work with cases $p=2$ and $1\le p<2$ separately.

\vspace{2 mm}

\noindent {\bf \underline{Case $p=2: \cl M\subseteq H^2(\bb D^n)$}} Let $\{\phi_1,\dots, \phi_r\}$ be an orthonormal set in 
$\cl W_{I_n}.$ Then, for each $i, \ \{\phi_i B_1^{m_1}\cdots B_n^{m_n}: m_1,\dots, m_n\ge 0\}$ is an orthonormal set in $\cl M.$ Thus, using Lemma \ref{STL2}, we obtain that 
\begin{enumerate}[(\em i)]
\item $\phi_i H^2(B_1,\dots, B_n)\subseteq \cl M$ for each $i$;
\item $\phi_i\in H^\infty(\bb D^n)$ for each $i;$
\item $||\phi_i f||_{\cl M}= ||f||_2$ for every $f\in H^2(B_1,\dots, B_n).$
\end{enumerate}  

Also, $\phi_i H^2(B_1,\dots, B_n)\perp \phi_j H^2(B_1,\dots, B_n)$ whenever $i\ne j.$ Then, using Lemma \ref{STL3}  we conclude, $r\le r_1\cdots r_n.$ Hence the 
dimension of $\cl W_{I_n}$ in this case can at be most be $r_1\cdots r_n.$

 \vspace{2 mm}

\noindent{\bf \underline{Case: $1\le p<2$}} Using equation (\ref{rep2*}) and Lemma \ref{STL2}, we get that 
$\cl W_{I_n}\subseteq H^\frac{2p}{2-p}(\bb D^n)\subseteq H^2(\bb D^n)$ which implies $\cl M\cap H^2(\bb D^n)\ne \emptyset$. 

Set $\cl X=\cl M\cap H^2(\bb D^n).$ Then we can easily check that  
$$
||x||^2 = ||x||_{\cl M}^2 + ||x||_2^2
$$
defines a norm on $\cl X$ and it becomes a Hilbert space with this norm. Note that $\cl X\subseteq H^2(\bb D^n)$ and each $T_i$ is an isometry on $\cl X.$ Hence using Case $p=2,$ we get 
$$
\cl X=\bigoplus\limits_{m_1,\dots, m_n\ge 0} T_1^{m_1}\cdots T_n^{m_n}\cl N
$$ 
where $\cl N=\cl X_1\cap\cdots\cap\cl X_n$ with $\cl X_i=ker(T_i^*)$ and $dim(\cl N)$ can at most be $r_1\cdots r_n.$ We will show that $dim(\cl W_{I_n})\le dim(\cl N).$ For this, let $\{f_1,\dots, f_k\}$ be a linearly independent set in $\cl W_{I_n}.$ Note that, using Proposition \ref{wdinv-prop}, we can decompose $\cl X$ in terms of $\cl X_1,\dots, \cl X_{n-1}$ as follows:
$$
\cl X= T_1(\cl X)\oplus T_2(\cl X_1)\oplus T_3(\cl X_1\cap \cl X_2)\oplus\cdots \oplus T_n\left(\bigcap_{i=1}^{n-1}\cl X_i\right)\oplus \left(\bigcap_{i=1}^n \cl X_i\right).
$$
Since $\cl W_{I_n}\subseteq X,$ we can represent each $f_i$ as 
$$
f_i=T_1(a_i^1)+T_2(a_i^2)+T_3(a_i^3)+\cdots+T_n(a_i^n)+b_i,
$$
where $a_i^1\in \cl X, \ b_i\in \cl X_1\cap\cdots\cap \cl X_n=\cl N$ and $a_i^j\in \cl X_1\cap\cdots \cap \cl X_{j-1}$ for $2\le j\le n.$  
We claim that 
$\{b_1,\dots, b_k\}$ is a linearly independent set. Suppose $\sum\limits_{i=1}^k\alpha_i b_i=0.$ Then,
\begin{equation}\label{ST12}
\sum_{i=1}^k \alpha_i f_i =\sum_{i=1}^k \sum_{j=1}^n \alpha_iT_j(a_i^j)= \sum_{j=1}^n T_j\left(\sum_{i=1}^k \alpha_ia_i^j\right) .
\end{equation}
Note that the right hand side of equation (\ref{ST12}) has an element of 
$\sum_{i=1}^n T_i(\cl X)\subseteq \sum_{i=1}^n T_i(\cl M)\subseteq \cl W_{I_n}^\perp,$ whereas left hand side is an element of 
$\cl W_{I_n}.$ Therefore, $\sum_{i=1}^k \alpha_if_i=0$ which implies that each 
$\alpha_i=0.$ This proves that $\{b_1,\dots,b_k\}$ is a linearly independent subset of $\cl N.$ Hence $k$ can at be most $r_1\cdots r_n,$ since $dim(\cl N)\le r_1\cdots r_n.$ This implies that $dim(\cl W_{I_n})\le r_1\cdots r_n.$

\vspace{.2 cm}

From the above two cases, we conclude that $dim(\cl W_{I_n})\le r_1\cdots r_n, \ \cl W_{I_n}\subseteq H^\infty(\bb D^n) $ when $\cl M$ is a vector subspace of $H^2(\bb D^n),$ and $\cl W_{I_n}\subseteq H^\frac{2p}{2-p}(\bb D^n)$ when $\cl M$ is boundedly contained in $H^p(\bb D^n).$ Now equation (\ref{rep3}) follows from equation (\ref{rep2*}) by taking an orthonormal basis for $\cl W_{I_n}.$ Lastly, equation 
(\ref{rep4}) follows from equation (\ref{rep3}) together with the fact $||\phi f||_{\cl M}=||f||_2$ for every 
$\phi\in \cl M$ with $||\phi||_{\cl M}$ and $f\in H^2(B_1,\dots, B_n).$ This completes the proof.
\qedsymbol

\vspace{1 mm}

\begin{lemma}\label{uniq} Let $\cl Y$ be a Hilbert space space consisting of analytic functions over $\bb D^n.$ Suppose $\{\phi_1,\dots, \phi_r\}$ is a linearly independent set in $\cl  Y$ such that
\begin{enumerate}[(i)]
\item for each $1\le j\le r, \ \phi_jf\in \cl Y$ whenever $f\in H^2(B_1,\dots,B_n)$ and 
$\phi_jf = \sum\limits_{\substack{{\bs m} =(m_1,\dots, m_n)\\m_1,\dots, m_n=0}}^{\infty}\alpha_{\bs m}\phi_jB_1^{m_1}\cdots B_n^{m_n}$ 
for $f=\sum\limits_{\substack{{\bs m} =(m_1,\dots, m_n)\\m_1,\dots, m_n=0}}^{\infty}\alpha_{\bs m}B_1^{m_1}\cdots B_n^{m_n};$ 
\item $\phi_j B_1^{l_1}\cdots B_n^{l_n}\perp \phi_k B_1^{m_1}\cdots B_n^{m_n}$ in $\cl Y$ for every $1\le j, \ k \le r$ whenever 
$(l_1,\dots,l_n)\ne (m_1,\dots,m_n).$
\end{enumerate}
Then $\phi_1f_1+\cdots+\phi_rf_r=0$ for $f_1,\dots, f_r\in H^2(B_1,\dots,B_n)$ if and only if each $f_i=0.$
\end{lemma}
\begin{proof} Suppose $f_1,\dots,f_r\in H^2(B_1,\dots,B_n)$ such that $\phi_1f_1+\cdots+\phi_rf_r=0.$ Then, for any fixed ${\bs m}=(m_1,\dots, m_n)$ 
and $1\le j\le r,$ 
$$ 
\langle{\phi_1f_1+\cdots+\phi_r f_r,\phi_j B_1^{m_1}\cdots B_n^{m_n}}\rangle =0.
$$

Let $f_i = \sum\limits_{\substack{{\bs k} =(k_1,\dots, k_n)\\k_1,\dots, k_n=0}}^{\infty} \alpha_{\bs k}^{i}B_1^{k_1}\cdots B_n^{k_n}$ for $1\le i\le r.$ Then 
$$
\left\langle{\alpha_{\bs m}^1B_1^{m_1}\cdots B_n^{m_n}\phi_1+\cdots+\alpha_{\bs m}^r\phi_r B_1^{m_1}\cdots B_n^{m_n},\phi_j B_1^{m_1}\cdots B_n^{m_n}}\right\rangle=0.  
$$

Thus, 
$$
\left\|(\alpha_{\bs m}^1\phi_1+\cdots +\alpha_{\bs m}^r\phi_r)B_1^{m_1}\cdots B_n^{m_n}\right\|^2 =0
$$
which implies 
$\alpha_{\bs m}^1\phi_1+\cdots +\alpha_{\bs m}^r\phi_r=0.$ Now $\phi_1,\dots, \phi_r$ are linearly independent, therefore 
$\alpha_{\bs m}^1=\cdots=\alpha_{\bs m}^r=0.$ By repeating this process for all ${\bs m}=(m_1,\dots,m_n),$ we get $\alpha_{\bs m}^j=0$ for every 
${\bs m}$ and every $1\le j\le r.$ Hence, each $f_j=0.$ This completes the proof.
\end{proof}

\begin{thm}\label{Hp-bdd-sv} Let $\cl M$ be a non-zero Hilbert space that is a vector subspace of $H^p(\bb D^n)$ for some $1\le p\le 2.$ Suppose 
$\cl M$ is boundedly contained in $H^p(\bb D^n)$ when $1\le p< 2.$ Further, suppose the operators $T_1,\dots, T_n$ satisfy the following conditions on $\cl M:$
\begin{enumerate}[(i)]
\item $T_j\cl M\subseteq \cl M$ for each $j$;
\item each $T_j$ is a near-isometry on $\cl M$;
\item the tuple $(T_1,\dots,T_n)$ doubly commute on $\cl M.$
\end{enumerate}

Then there exists an orthonormal set $\{\phi_1,\dots, \phi_r\}$ in $\cl M$ with $r\le r_1\dots r_n$ such 
that 
$$
\cl M = \overline {\phi_1H^2(B_1,\dots, B_n)\oplus \cdots \oplus \phi_r H^2(B_1,\dots, B_n)}
$$
where  the closure is in the norm on $\cl M$ and 
\begin{enumerate}[(a)]
\item each $\phi_i\in H^{\frac{2p}{2-p}}(\bb D^n)$ with the understanding that $\frac{2p}{2-p}$ means $\infty$ when $p=2,$ 
\item $||\phi_if||_{\cl M}\le ||f||_{\cl M}$ for each $f\in H^2(B_1,\dots,B_n)$ and $1\le i\le r$, 
\item the direct sum is an algebraic direct sum.
 \end{enumerate}
\end{thm}

\begin{proof} Since the tuple $(T_1,\dots, T_n)$ doubly commute and each $T_i$ is a near-isometry on $\cl M,$ using Theorem \ref{wdinvsv}, 
we decompose $\cl M$ as  
\begin{equation}\label{rep10}
\cl M = \bigoplus_{A\subseteq I_n} \cl M_A,
\end{equation}
where 
\begin{align*}
\cl M_A&=\bigoplus_{{\bs k}\in \bb N_0^A} T^{\bs k}_A\left(\bigcap_{{\bs j}\in \bb N_0^{I_n\setminus A}}T^{\bs j}_{I_n\setminus A}(\cl W_A)\right) \quad when 
\quad  \emptyset \ne A\subseteq I_n \ and \\
 \cl M_A &= \bigcap_{{\bs k}\in \bb N_0^{I_n}} T^{\bs k}_{I_n }\cl M \quad when \quad A=\emptyset
\end{align*}
with $\cl W_A=\bigcap_{i\in A}\cl W_i, \ \cl W_i=ker(T_i^*).$ Notice that elements of $\cl M$ are analytic functions on $\bb D^n$ and 
$B_i(0)=0$ for all $1\le i\le n$, therefore $\cl M_A=\{0\}$ whenever $I_n\setminus A \ne \emptyset.$ So, equation (\ref{rep10}) reduces to  
\begin{equation}\label{rep2}
\cl M = \bigoplus_{{\bs k}\in I_n}T^{\bs k}_{I_n}\cl W_{I_n}
\end{equation}

We will first show that $\cl W_{I_n}\subseteq H^\frac{2p}{2-p}(\bb D^n).$ For this, take $\phi\in \cl W_{I_n}$ with $||\phi||_{\cl M}=1$. Then $\{\phi B_1^{m_1}\cdots B_n^{m_n}: m_1, \dots, m_n\ge 0\}$ is an orthonormal set in $\cl M.$ Suppose $f\in H^2(B_1,\dots, B_n)$ with $f=\sum\limits_{\substack{{\bs j} =(j_1,\dots, j_n)\\j_1,\dots, j_n=0}}^{\infty} \alpha_{\bs j}B_1^{j_1}\cdots B_n^{j_n}.$ Set 
 $f_{k_1\cdots k_n}=\sum\limits_{\substack{{\bs j} =(j_1,\dots, j_n)\\j_1,\dots, j_n=0}}^{k_1,\dots, k_n}\alpha_{\bs j}B_1^{j_1}\cdots B_n^{j_n},$ then the sequence $\{f_{k_1\cdots k_n}\}$ converges to $f$ in $H^2(\bb D^n).$ Now consider
\begin{align}
||\phi f_{k_1\dots k_n}||_{\cl M}^2 &= \sum\limits_{\substack{{\bs j}=(j_1,\dots, j_n)\\ j_1,\dots, j_n = 0}}^{k_1,\dots, k_n}||\alpha_{\bs j}\phi B_1^{j_1}\cdots B_n^{j_n} ||^2\nonumber\\
&\le  \sum\limits_{\substack{{\bs j}=(j_1,\dots, j_n)\\ j_1,\dots, j_n = 0}}^{k_1,\dots, k_n}|\alpha_{\bs j}|^2\nonumber\\
&= ||f_{k_1\cdots k_n}||^2_2\label{ST14}
\end{align}

This yields that $\{\phi f_{k_1\cdots k_n}\}$ is a Cauchy sequence in $\cl M$, hence converges to some $h$ in $\cl M.$ 
For $p=2,$ we deduce $\phi f=h$ by calculating the $({\bs i},{\bs k})^{th} \ B$-Fourier coefficient of $h$ using exactly the same arguments as we used in Case $p=2$ of Lemma \ref{STL2}. When $1\le p<2,$ $\cl M$ is boundedly contained in $H^p(\bb D^n)$ which implies that $\{\phi f_{k_1\dots k_n} \}$ converges to $h$ in $H^p(\bb D^n).$  Also, $f_{k_1\dots k_n}$ converges to $f$ in $H^2(\bb D^n).$ Thus  
$\phi f=h,$ since convergence in $p$-norm and 
$2-$norm implies pointwise convergence. Therefore, $\phi f = h$ irrespective of the value of $p.$ 

Hence, $\phi H^2(B_1,\dots , B_n)\subseteq \cl M\subseteq H^p(\bb D^n)$ for $1\le p\le 2.$ This, as we noted earlier as well, readily establishes that 
$\phi H^2(\bb D^n)\subseteq H^p(\bb D^n).$     
Consequently, $\cl W_{I_n}\subseteq H^{2p/{2-p}}(\bb D^n),$ where $\frac{2p}{2-p}$ means infinity when $p=2.$ Also, equation (\ref{ST14}) yields 
$||\phi f||_{\cl M}\le ||\phi||_{\cl M}||f||_2$ whenever $\phi \in \cl W_{I_n}$ and $f\in H^2(B_1,\dots,B_n)$.  

\vspace{.3 mm}

We will now show that the dimension of $\cl W_{I_n}$ can at most be $r_1\cdots r_n.$ To this end, let $\{\phi_1,\dots, \phi _r \}$ be an orthonormal set in $\cl W_{I_n}.$ 
Then, 
$$
\cl X=\phi_1H^2(B_1,\dots,B_n)+\cdots +\phi_r H^2(B_1,\dots, B_n)
$$
is a vector subspace of $H^p(\bb D^n).$ Note that the set $\{\phi_1,\dots, \phi_r\}\subseteq \cl M$ satisfies the hypotheses of Lemma \ref{uniq}, therefore,  
$\phi_1f_1+\cdots +\phi_rf_r=0$ for $f_1,\dots, f_r\in H^2(B_1,\dots, B_n)$ if and only if  
$f_1=\cdots =f_r=0.$ As a result,
$$
||\sum_{i=1}^r \phi_if_i||_{\cl X}^2 = \sum_{i=1}^r ||f_i||_2^2
$$
 defines a well-defined norm on $\cl X.$ In fact, $\cl X$ is a Hilbert space with respect to this norm.   
%
%
 
 Further, for any $f_1,\dots, f_r\in H^2(B_1,\dots,B_n)$  
 \begin{align*}
 ||\phi_1f_1+\cdots +\phi_r f_r ||_p &\le ||\phi_1f_1||_p+\cdots + ||\phi_r f_r||_p\\
&\le ||\phi_1||_{\frac{2p}{2-p}} ||f_1 ||_2+\cdots +||\phi_r||_{\frac{2p}{2-p}} ||f_r||_2\\
\end{align*}
Hence, $\cl X$ is a Hilbert space that is boundedly contained in $H^p(\bb D^n).$ Also, each $T_i$ is an isometry on $\cl X.$ Therefore, using Theorem \ref{STThm}, 
$dim(\cl X_1\cap\cdots\cap \cl X_n)$ can at most be $r_1\cdots r_n,$ where $\cl X_i=\cl X\ominus T_i{\cl X}.$ We can easily see that $\{\phi_1,\dots, \phi_r\}$ is an orthonormal set in $\cl X.$ Consequently, to show that $r\le r_1\cdots r_n$, it is enough to show that each $\phi_i\in \bigcap_{j=1}^n {\cl X}_j.$ Fix any $1\le j\le n,$ and let $f\in T_j{\cl X}.$ Set 
$f=T_j(\sum_{i=1}^r \phi_i f_i)$ for some $f_1,\dots, f_r\in H^2(B_1,\dots,B_n).$ Then, 
$$\langle{\phi_i,T_jf}\rangle_{\cl X}=\langle{\phi_i, \phi_1 B_j f_1+\cdots+\phi_r B_j f_r}\rangle_{\cl X} = \langle{1,B_jf_i}\rangle_2=0
$$
as $B_j(0)=0.$ Thus, $\phi_i\perp T_j{\cl X}$ which means $\phi_i\in {\cl X}_j.$ Thus, $\{\phi_1,\dots, \phi_r\}$ is an orthonormal set in 
$\bigcap_{j=1}^n {\cl X}_j$ which implies that $r\le r_1\cdots r_n.$ Hence, dimension of $\cl W_{I_n}$ can at most be $r_1\cdots r_n.$  

Lastly, let $\{\phi_1,\cdots, \phi_r\}$ be an orthonormal basis of $\cl W_n.$ Then, using equation (\ref{rep2}), we get that 
\begin{equation}\label{direct}
\phi_1H^2(B_1,\dots,B_n)+\cdots+\phi_rH^2(B_1,\dots,B_n)
\end{equation}
is dense in $\cl M.$ Furthermore, using Lemma \ref{uniq} again, we conclude that the sum on the right hand side of equation (\ref{direct}) is in fact an algebraic direct sum. Hence,
$$
\cl M = \overline {\phi_1H^2(B_1,\dots, B_n)\oplus \cdots \oplus \phi_r H^2(B_1,\dots, B_n)}.
$$
Finally, we note that the assertions $(a)$ and $(b)$ from the statements of the theorem have already been shown to hold true for any unit vector in $\cl W_{I_n},$ in particular, they are true for $\phi_i's.$ This completes the proof.
\end{proof}

\begin{remark}\label{n=1} In Theorem \ref{Hp-bdd-sv}, the bounded containment of the sub-Hardy Hilbert space $\cl M$ in 
$H^p(\bb D^n)$ for $1\le p<2$  is used only to show that the subspace $\cl W_{I_n}$ is a subset of $H^\frac{2p}{2-p}(\bb D^n).$ We could work without this condition for 
$p=2$ case because of equation (\ref{ST7}) which is a decomposition of $H^2(\bb D^n)$ in terms of $H^2(B_1, \dots, B_n).$ A similar decomposition exists for $H^p(\bb D), \ 1\le p\le \infty$ which for $p=2$ is indeed equation (\ref{ST7}) with $n=1.$ We brief it here for reader's convenience. Suppose $B$ is a finite Blaschke product with $k$ factors and $e_{i0}, \ 1\le i\le k-1$ are functions defined by equation (\ref{onb}) corresponding to $B.$ Then \cite{LaSt} proves that   
$$
H^p(\bb D) =e_{00}H^p(B)\oplus \cdots\oplus e_{(k-1)0}H^p(B),
$$ 
where $H^p(B)$ is the closed linear span of $\{B^j:j\ge 0\}$ in $H^p(\bb D)$ and the sum is an algebraic direct sum for $p\ne 2.$ It is 
not difficult to check that $H^p(B)=\{f\circ B:f\in H^p(\bb D)\}.$ Now, by using this decomposition and the fact that elements of $H^p(\bb D)$ are power series on the unit disk we can easily show that elements of $\cl W_{I_1}$ multiplies $H^2(B)$ into $\cl M$ without assuming a bounded containment of $\cl M$ into $H^p(\bb D).$ Therefore, the bounded containment condition is redundant for one variable case in Theorem \ref{Hp-bdd-sv}. 
\end{remark}

\begin{thm}\label{p>2} Let $\cl M$ be a Hilbert space which is a vector subspace of $H^p(\bb D^n)$ for some $p>2.$ Suppose the tuple $(T_1,\dots, T_n)$ satisfies the conditions $(i), \ (ii),$ and $(iii)$ of Theorem \ref{Hp-bdd-sv}. Then $\cl M=\{0\}.$ 
\end{thm}
\begin{proof} Since $p>2, \ \cl M$ is a vector subspace of $H^2(\bb D^n)$ as well. Thus, using the same arguments as we used in Theorem 
\ref{Hp-bdd-sv} for Case $p=2,$ we get

$$
\cl M = \bigoplus_{r\in I_n}T^r_{I_n}\cl W_{I_n}
$$
and functions in $\cl W_{I_n}$ multiplies $H^2(B_1,\dots,B_n)$ into $\cl M\subseteq H^p(\bb D^n).$ This implies that functions in $\cl W_{I_n}$ multiplies $H^2(\bb D^n)$ into 
$H^p(\bb D^n)\subseteq H^2(\bb D^n),$ hence $\cl W_{I_n}\subseteq H^\infty(\bb D^n).$ But no non-zero function in $H^\infty(\bb D^n)$ can multiply $H^2(\bb D^n)$ into $H^p(\bb D^n).$ Hence $\cl W_{I_n}=\{0\}$ which implies $\cl M=\{0\}.$ This completes the proof.
\end{proof}

\begin{cor}\label{LS}\cite[Theorem 3.1]{LS} Let $\cl M$ be a Hilbert space which is a vector subspace of $H^2(\bb D).$ Suppose the operator of 
multiplication with the coordinate function $z$ is well-defined on $\cl M$ and is a near-isometry on $\cl M$. Then there exists 
a function $\phi\in \cl M\cap H^\infty(\bb D)$ with 
$||\phi||_{\cl M}=1$ such that 
$$
\cl M= \overline{\phi H^2(\bb D)} \quad ({\rm closure \ is \ in \ the \ norm \ on \ \cl M })
$$
and $||\phi f||_{\cl M}\le ||f||_{2}$ for all $f\in H^2(\bb D).$
\end{cor}
\begin{proof}This is simply the one-variable case of Theorem \ref{Hp-bdd-sv} for $p=2$ where the Blaschke product is the coordinate function $z.$ 
\end{proof}

\begin{cor}\cite[Theorem 4.1]{ST} Let $\cl M$ be a Hilbert space such that 
\begin{enumerate}[(i)]
\item $\cl M$ is algebraically contained in $H^2(\bb D);$
\item $\cl M$ is invariant under $T_B,$ the operator of multiplication with a finite Blaschke product $B$ with $m$ factors;
\item $T_B$ is an isometry on $B.$
\end{enumerate}
Then $$
\cl M=\phi_1H^2(B)\oplus \cdots \oplus \phi_r H^2(B)
$$
where each $\phi_j\in H^\infty, \ 1\le j\le r,$ and $r\le m.$ Further,
$$
||\phi_1f_1+\cdots +\phi_r f_r||_{\cl M}^2 = ||f_1||_2^2+\cdots +||f_r||_2^2, \quad f_j\in H^2(B)
$$
for each $f=\phi_1f_1+\cdots+\phi_r f_r$ in $\cl M.$
\end{cor}
\begin{proof} Every isometry is a near-isometry and hence the result follows form Theorem \ref{Hp-bdd-sv} by taking $n=1$ and $p=2.$
\end{proof}

\subsection*{Acknowledgements} The first and third authors
thank the Mathematical Sciences Foundation,
Delhi for support and facilities needed to complete
the present work. The research of first author is supported by the Mathematical Research Impact Centric Support (MATRICS) grant, File No: MTR/2017/000749, by the Science and Engineering Research Board (SERB), Department of Science \& Technology (DST), Government of India.

\end{document}